\documentclass[12pt]{amsart}
\usepackage{fullpage}
\usepackage{amssymb}
\newtheorem{theorem}{Theorem}
\newtheorem{proposition}[theorem]{Proposition}
\newtheorem{lemma}[theorem]{Lemma}
\newcommand{\Z}{\mathbb{Z}}
\newcommand{\Q}{\mathbb{Q}}
\newcommand{\R}{\mathbb{R}}
\newcommand{\C}{\mathbb{C}}
\newcommand{\A}{\mathbb{A}}
\renewcommand{\H}{\mathbb{H}}
\DeclareMathOperator{\sgn}{sgn}

\DeclareMathOperator{\GL}{GL}
\begin{document}
\title{Simple zeros of degree $2$ $L$-functions}
\author{Andrew R.~Booker}
\address{School of Mathematics, University of Bristol, University Walk,
Bristol, BS8 1TW, United Kingdom}
\thanks{The author was supported by EPSRC Fellowship EP/H005188/1.}
\begin{abstract}
We prove that the complete $L$-functions of classical holomorphic newforms
have infinitely many simple zeros.
\end{abstract}
\maketitle
\section{Introduction}
Let $\pi$ be a cuspidal automorphic representation of $\GL_n(\A_\Q)$
with corresponding $L$-function $\Lambda(s,\pi)$.  The Grand
Riemann Hypothesis (GRH) and Grand Simplicity Hypothesis (GSH) predict
that the zeros of $\Lambda(s,\pi)$ lie on the line $\Re(s)=\frac12$ and
are simple, apart from at most one multiple zero if $\pi$ is associated
to a geometric motive (cf.\ the BSD conjecture).  These conjectures have
not yet been shown to hold for a single example, and most partial evidence
in their favor has been for $n=1$, i.e.\ the Dirichlet $L$-functions. In
particular, for $n>1$, the only cuspidal representation that we are
aware of for which it is known that $\Lambda(s,\pi)$ has infinitely many
simple zeros is the one associated to the Ramanujan $\Delta$ modular form,
which is a theorem of Conrey and Ghosh \cite{conrey-ghosh} from 1988.

As Conrey and Ghosh remark in their paper, most of their arguments would
apply to any degree $2$ $L$-function, but they were unable to conclude
the proof without assuming \emph{a priori} the existence of at least one
simple zero (which they verified directly for the $L$-function associated
to $\Delta$). In this paper, we analyze their method from a structural
point of view, along the lines of \cite{bk-weil} and \cite{kp},
to prove the following:
\begin{theorem}
\label{mainthm}
Let $f\in S_k(\Gamma_1(N))^{\rm new}$ be a normalized Hecke eigenform
of arbitrary weight and level.  Then the complete
$L$-function $\Lambda_f(s)=\int_0^\infty f(iy)y^{s-1}\,dy$ has infinitely
many simple zeros.
\end{theorem}
As our proof will show, a lack of simple zeros leads to inconsistencies
unless the local $L$-factor of $\Lambda_f(s)$ is a square at every
unramified prime (which cannot happen for holomorphic modular forms). In
effect, we establish a connection (albeit a very loose one) between the
zeros of the global $L$-function and those of its local factor
polynomials.

Recently, Cho \cite{cho} has generalized \cite{conrey-ghosh} to prove
that the $L$-functions of the first few Maass cusp forms of level $1$
have infinitely many simple zeros.  Our proof could be modified in
an analogous fashion to extend Theorem \ref{mainthm} to all cuspidal
Maass newforms.  Moreover, the assumption that $f$ is a cusp form is also
unnecessary, so in fact the method could be generalized to show that if
$\chi_1$ and $\chi_2$ are primitive Dirichlet characters and $t\in\R$
then $\Lambda(s,\chi_1)\Lambda(s+it,\chi_2)$ has infinitely many simple
zeros unless $\chi_1=\chi_2$ and $t=0$.  However, stronger results of
this type may be obtained by other methods, e.g.\ \cite{cgg}.

Finally we note that Conrey and Ghosh's result for $f=\Delta$
is a bit stronger than the conclusion of Theorem~\ref{mainthm} for
that case.  Precisely, if $N_f(T)$ denotes the number of simple zeros
of $\Lambda_f(s)$ with imaginary part in $[0,T]$, they showed that
for every $\varepsilon>0$, the inequality $N_\Delta(T)\ge
T^{\frac16-\varepsilon}$ holds for some arbitrarily large values of $T$.
With Theorem~\ref{mainthm} in hand, it seems likely that their proof
of this estimate would generalize at least to all eigenforms of level $1$.
However, in this paper we content ourselves with the qualitative statement
of Theorem~\ref{mainthm}.

\subsection*{Acknowledgements}
I thank Brian Conrey for helpful comments.

\subsection*{Notation}
Let $f$ be as in the statement of Theorem~\ref{mainthm}, and let
$\xi$ denote its nebentypus character.
Let
$$
L_f(s)=\sum_{n=1}^\infty a_f(n)n^{-s}
=\prod_p\frac1{1-a_f(p)p^{-s}+\xi(p)p^{k-1-2s}}
$$
be the finite $L$-function of $f$, and
$\Lambda_f(s)=(2\pi)^{-s}\Gamma(s)L_f(s)$ the completed version.
Then we have the functional equation
\begin{equation}
\label{funceq}
\Lambda_f(s)=\epsilon N^{\frac{k}2-s}\Lambda_{\bar{f}}(k-s),
\end{equation}
where $\bar{f}\in S_k(\Gamma_0(N),\overline{\xi})$ is the
dual of $f$, and $\epsilon\in\C$ is the root number.
We define
$$
D_f(s)=L_f(s)\frac{d^2}{ds^2}\log L_f(s)=\sum_{n=1}^{\infty}c_f(n)n^{-s}.
$$
Note that $D_f(s)$
continues meromorphically to $\C$, with poles precisely at the simple
zeros of $L_f(s)$ (including the trivial zeros $s=0, -1, -2, \ldots$).

Next, for any $\alpha\in\Q^\times$, we define the additive twists
$$
L_f(s,\alpha)=\sum_{n=1}^\infty a_f(n)e(\alpha n)n^{-s}
\quad\mbox{and}\quad
D_f(s,\alpha)=\sum_{n=1}^\infty c_f(n)e(\alpha n)n^{-s}.
$$
By Deligne's bound $|a_f(p)|\le 2p^{\frac{k-1}2}$, we see that
each of these is holomorphic for $\Re(s)>\frac{k+1}2$.
Moreover, it follows from \cite[Prop.~3.1]{bk-converse} that
$L_f(s,\alpha)$
continues to an entire function.
One could similarly prove that $D_f(s,\alpha)$ has meromorphic
continuation to $\C$ for every $\alpha$,
but it turns out to be enough for our purposes to 
consider $\alpha=1/q$, where $q$ is a prime number not dividing
$N$.  In this case, we have the following expansion of the exponential
function in terms of Dirichlet characters:
$$
e\!\left(\frac{n}{q}\right)
=1-\frac{q}{q-1}\chi_0(n)
+\frac1{q-1}\sum_{\substack{\chi\;(\text{mod }q)\\\chi\ne\chi_0}}
\tau(\overline{\chi})\chi(n),
$$
where $\chi_0\;(\text{mod }q)$ is the trivial character,
the sum ranges over all non-trivial $\chi\;(\text{mod }q)$,
and $\tau(\overline{\chi})$ denotes the Gauss sum of $\overline{\chi}$.
Multiplying both sides by $c_f(n)n^{-s}$ and summing over $n$, we thus see that
$$
D_f\!\left(s,\frac1q\right)=D_f(s)
-\frac{q}{q-1}D_f(s,\chi_0)
+\frac1{q-1}\sum_{\substack{\chi\;(\text{mod }q)\\\chi\ne\chi_0}}
\tau(\overline{\chi})D_f(s,\chi),
$$
where, for each $\chi$, $D_f(s,\chi)$ denotes the multiplicative twist
$$
D_f(s,\chi)=\sum_{n=1}^\infty c_f(n)\chi(n)n^{-s}.
$$

By the known non-vanishing results
for automorphic $L$-functions \cite{jacquet-shalika},
all poles of $D_f(s)/\Gamma(s)$ and $D_f(s,\chi)/\Gamma(s)$ for
$\chi\ne\chi_0$ are confined to the
critical strip
$\bigl\{s\in\C:\Re(s)\in\bigl(\frac{k-1}2,\frac{k+1}2\bigr)\bigr\}$.
On the other hand, from the formula
$$
\sum_{n=1}^{\infty}a_f(n)\chi_0(n)n^{-s}
=(1-a_f(q)q^{-s}+\xi(q)q^{k-1-2s})L_f(s),
$$
it follows that $D_f(s,\chi_0)$ has a pole at every simple zero of the
local Euler factor polynomial $1-a_f(q)q^{-s}+\xi(q)q^{k-1-2s}$, except
possibly at $s=0$ when $k=1$. By Deligne, the zeros of this polynomial
occur on the line $\Re(s)=\frac{k-1}2$, and they are simple if and only
if the polynomial is not a square.  By the above, we see that $D_f(s,1/q)$
inherits these poles when they occur.

\section{Proof of Theorem~\ref{mainthm}}
The main tool used in the proof is the following proposition, whose
proof we defer until the final section.
\begin{proposition}
\label{mainprop}
Suppose that $\Lambda_f(s)$ has at most finitely many simple
zeros.  Then, for any $\alpha\in\Q^\times$ and $M\in\Z_{\ge0}$,
\begin{equation}
\label{maineq}
\begin{aligned}
D_f(s,\alpha)
-\epsilon(i\sgn(\alpha))^k
(N\alpha^2)^{s-\frac{k}2}
\sum_{m=0}^{M-1}
&m!\left(\frac{iN\alpha}{2\pi}\right)^m
{{s+m-1}\choose{m}}
{{s+m-k}\choose{m}}\\
&\cdot D_{\bar{f}}\!\left(s+m,-\frac1{N\alpha}\right)
\end{aligned}
\end{equation}
continues to a holomorphic function for $\Re(s)>\frac{k+1}2-M$.
\end{proposition}

From now on we will assume that $\Lambda_f(s)$ has at most finitely
many simple zeros and attempt to reach a contradiction.  To that end,
let $M$ be a positive integer and $q$ a prime not dividing $N$.
By Dirichlet's theorem, there are
distinct primes $q_1,\ldots,q_M\in q+N\Z$, and it follows that
$D_{\bar{f}}(s,-q_j/N)=D_{\bar{f}}(s,-q/N)$ for every $j$.
Thus, applying Prop.~\ref{mainprop} with $\alpha=1/q_j$, we obtain that
\begin{equation}
\label{qjhol}
\left(\frac{N}{q_j^2}\right)^{\frac{k}2-s}
D_f\!\left(s,\frac1{q_j}\right)
-\epsilon i^k\sum_{m=0}^{M-1}
m!\left(\frac{iN}{2\pi q_j}\right)^m
{{s+m-1}\choose{m}}
{{s+m-k}\choose{m}}
D_{\bar{f}}\!\left(s+m,-\frac{q}{N}\right)
\end{equation}
is holomorphic for $\Re(s)>\frac{k+1}2-M$.

Next let $m_0\in\Z$ with $0\le m_0<M$. By the Vandermonde determinant,
there are numbers $c_1,\ldots,c_M\in\Q$ such that
$$
\sum_{j=1}^M c_jq_j^{-m}=
\begin{cases}
1&\text{if }m=m_0,\\
0&\text{if }m\ne m_0
\end{cases}
\quad\text{for every }
m\in\Z\cap[0,M).
$$
Multiplying \eqref{qjhol} by $-c_j$, summing over $j$ and
replacing $s$ by $s-m_0$, we find that
$$
\epsilon i^k
m_0!\left(\frac{iN}{2\pi}\right)^{m_0}
{{s-1}\choose{m_0}}
{{s-k}\choose{m_0}}
D_{\bar{f}}\!\left(s,-\frac{q}{N}\right)
-\sum_{j=1}^Mc_j
\left(\frac{N}{q_j^2}\right)^{\frac{k}2+m_0-s}
D_f\!\left(s-m_0,\frac1{q_j}\right)
$$
is holomorphic for $\Re(s)>m_0+\frac{k+1}2-M$.
This establishes the meromorphic continuation of $D_{\bar{f}}(s,-q/N)$
to that region.
Moreover, since $D_f(s,1/q_j)$ is holomorphic
on $\{s\in\C:\Re(s)<\frac{k-1}2\}\setminus\Z$ for each $j$, we see that
$D_{\bar{f}}(s,-q/N)$ is holomorphic
on $\{s\in\C:\Re(s)\in(m_0+\frac{k+1}2-M,m_0+\frac{k-1}2)\}\setminus\Z$.
Thus, choosing $m_0=2$ and $M$ arbitrarily large, we find that
$D_{\bar{f}}(s,-q/N)$ has meromorphic continuation to $\C$, with poles
possible only at integer points.

Hence, applying Prop.~\ref{mainprop} again with $\alpha=1/q$ and
$M=2$, we learn that $D_f(s,1/q)$ can only have poles at integer
points.  However, we have already seen that $D_f(s,1/q)$ has a pole
at every simple zero (except possibly $s=0$) of the local Euler factor
polynomial $1-a_f(q)q^{-s}+\xi(q)q^{k-1-2s}$. This polynomial, in turn, has
infinitely many simple zeros along the line $\Re(s)=\frac{k-1}2$ if and
only if $|a_f(q)|<2q^{\frac{k-1}2}$.  By the Rankin--Selberg method,
the average value of $|a_f(q)|^2/q^{k-1}$ is $1$, so such primes $q$
exist in abundance.  This concludes the proof of Theorem~\ref{mainthm}.

\section{Proof of Proposition~\ref{mainprop}}
Let $\Delta_f(s)=(2\pi)^{-s}\Gamma(s)D_f(s)$.
Taking the logarithm of \eqref{funceq} and differentiating twice, we find
$$
\psi'(s)+\frac{d^2}{ds^2}\log L_f(s)
=\psi'(k-s)+\frac{d^2}{ds^2}\log L_{\bar f}(k-s),
$$
where $\psi(s)=\frac{\Gamma'}{\Gamma}(s)$ is the digamma function.
Thus, it follows that
\begin{equation}
\label{dfunceq}
\Delta_f(s)+\Lambda_f(s)(\psi'(s)-\psi'(k-s))
=\epsilon N^{\frac{k}2-s}\Delta_{\bar{f}}(k-s).
\end{equation}
Next, since $\Lambda_f(s)$ has at most finitely many simple zeros,
there is a rectangle $\mathcal{C}$ contained within the critical
strip $\{s\in\C:\Re(s)\in(\frac{k-1}2,\frac{k+1}2)\}$ which encloses
all simple zeros. For $z\in\H=\{z\in\C:\Im(z)>0\}$, we define
$$
F(z)=\sum_{n=1}^\infty c_f(n)e(nz),\quad
\overline{F}(z)=\sum_{n=1}^\infty c_{\bar{f}}(n)e(nz),
$$
$$
A(z)=\frac1{2\pi i}\int_{\Re(s)=k-\frac12}
\bigl(\psi'(s)+\psi'(s+1-k)\bigr)\Lambda_f(s)(-iz)^{-s}\,ds,
$$
and
$$
B(z)=\frac1{2\pi i}\int_{\mathcal{C}}\Delta_f(s)(-iz)^{-s}\,ds
+\frac1{2\pi i}\int_{\Re(s)=k-\frac12}\frac{\pi^2}{\sin^2(\pi s)}
\Lambda_f(s)(-iz)^{-s}\,ds.
$$
Here $\mathcal{C}$ is given counter-clockwise orientation, and
$(-iz)^{-s}$ is defined as $e^{-s\log(-iz)}$ using the principal
branch of the logarithm.

These functions are related as follows:
\begin{lemma}
\label{mainidlemma}
We have
\begin{equation}
\label{mainidentity}
F(z)+A(z)
=\epsilon(-i\sqrt{N}z)^{-k}\overline{F}\!\left(-\frac1{Nz}\right)+B(z)
\end{equation}
for all $z\in\H$.
\end{lemma}
\begin{proof}
By Mellin inversion, we have
$$
F(z)=\frac1{2\pi i}\int_{\Re(s)=\frac{k}2+1}\Delta_f(s)(-iz)^{-s}\,ds
$$
and
$$
\epsilon(-i\sqrt{N}z)^{-k}\overline{F}\!\left(-\frac1{Nz}\right)
=\frac{\epsilon N^{k/2}}{2\pi i}
\int_{\Re(s)=\frac{k}2+1}\Delta_{\bar{f}}(s)(-iNz)^{s-k}\,ds.
$$
Since $\Lambda_f(s)$ has at most finitely many simple zeros,
there is a $\delta>0$ such that
$\Delta_{\bar{f}}(s)$ is holomorphic for $\Re(s)>\frac{k+1}2-\delta$.
Moreover, it follows from the Phragm\'en--Lindel\"of convexity principle
that for any fixed $z$, $\Delta_{\bar{f}}(s)(-iNz)^{s-k}$ decays rapidly
as $|\Im(s)|\to\infty$ in any fixed vertical strip.
Hence, we may shift the contour of the last line to
$\Re(s)=\frac{k+1-\delta}2$ and apply \eqref{dfunceq} to obtain
\begin{equation}
\label{fbarint}
\begin{aligned}
&\frac{\epsilon N^{k/2}}{2\pi i}
\int_{\Re(s)=\frac{k+1-\delta}2}\Delta_{\bar{f}}(s)(-iNz)^{s-k}\,ds
=\frac{\epsilon N^{k/2}}{2\pi i}
\int_{\Re(s)=\frac{k-1+\delta}2}
\Delta_{\bar{f}}(k-s)(-iNz)^{-s}\,ds\\
&=\frac1{2\pi i}\int_{\Re(s)=\frac{k-1+\delta}2}\Delta_f(s)(-iz)^{-s}\,ds
+\int_{\Re(s)=\frac{k-1+\delta}2}\Lambda_f(s)[\psi'(s)-\psi'(k-s)]
(-iz)^{-s}\,ds.
\end{aligned}
\end{equation}

Note that
$$
\frac1{2\pi i}\int_{\Re(s)=\frac{k}2+1}\Delta_f(s)(-iz)^{-s}\,ds
-\frac1{2\pi i}\int_{\Re(s)=\frac{k-1+\delta}2}\Delta_f(s)(-iz)^{-s}\,ds
=\frac1{2\pi i}\int_{\mathcal{C}}\Delta_f(s)(-iz)^{-s}\,ds,
$$
which is the first term of $B(z)$.
Next, since $\psi'(s)-\psi'(k-s)$ is holomorphic for
$\Re(s)\in(0,k)$, we may shift the contour of the last integral
in \eqref{fbarint} to $\Re(s)=k-\frac12$.
Using the reflection formula
$\psi'(1-s)+\psi'(s)=\pi^2/\sin^2(\pi s)$,
we have
$$
\psi'(s)-\psi'(k-s)=\psi'(s)+\psi'(s+1-k)-\frac{\pi^2}{\sin^2(\pi s)}.
$$
This yields $A(z)$ and the remaining term of $B(z)$.
\end{proof}

Now, the main idea of the proof of Prop.~\ref{mainprop}
is to compute $(2\pi)^s/\Gamma(s)$ times the Mellin
transform of both sides of \eqref{mainidentity} along the line
$\Re(z)=\alpha\in\Q^\times$. For $F(z)$, we have
\begin{equation}
\label{Fmellin}
\begin{aligned}
\frac{(2\pi)^s}{\Gamma(s)}\int_0^\infty F(\alpha+iy)y^s\frac{dy}y
&=\frac{(2\pi)^s}{\Gamma(s)}\int_0^\infty
\sum_{n=1}^\infty c_f(n)e(\alpha n)e^{-2\pi ny}y^s\frac{dy}y\\
&=\sum_{n=1}^\infty c_f(n)e(\alpha n)n^{-s}
=D_f(s,\alpha).
\end{aligned}
\end{equation}

\begin{lemma}
\label{Alemma}
For any $\alpha\in\Q^\times$,
$$
\frac{(2\pi)^s}{\Gamma(s)}\int_0^\infty A(\alpha+iy)y^s\frac{dy}y
$$
continues to an entire function of $s$.
\end{lemma}
\begin{proof}
Set $\Phi(s)=\psi'(s)+\psi'(s+1-k)$.
From the identity
$\psi'(s)=\int_1^\infty\frac{\log{x}}{x-1}x^{-s}\,dx$,
we get the integral representation
$\Phi(s)=\int_1^\infty\phi(x)x^{-s}\,dx$ for $\Re(s)>k-1$, where
$\phi(x)=\frac{(x^{k-1}+1)\log{x}}{x-1}$.
Hence
\begin{align*}
\Phi(s)\Gamma(s)
&=\int_1^\infty\phi(x)\int_0^\infty e^{-y}(y/x)^s\frac{dy}y\,dx
=\int_1^\infty\phi(x)\int_0^\infty e^{-xy}y^s\frac{dy}y\,dx\\
&=\int_0^\infty\int_1^\infty\phi(x)e^{-xy}\,dx\,y^s\frac{dy}y.
\end{align*}
Therefore, by Mellin inversion,
$$
A(z)=\frac1{2\pi i}\int_{\Re(s)=k+1}\Phi(s)\Gamma(s)
\sum_{n=1}^\infty a_f(n)(-2\pi inz)^{-s}\,ds
=\sum_{n=1}^\infty a_f(n)\int_1^\infty\phi(x)e(nxz)\,dx.
$$
Specializing to $z=\alpha+iy$, we get
$$
A(\alpha+iy)=\sum_{n=1}^\infty a_f(n)\int_1^\infty\phi(x)e(\alpha nx)
e^{-2\pi nxy}\,dx,
$$
so that
\begin{align*}
\int_0^\infty A(\alpha+iy)y^s\frac{dy}y
&=\sum_{n=1}^\infty a_f(n)\int_1^\infty\phi(x)e(\alpha nx)
\int_0^\infty e^{-2\pi nxy}y^s\frac{dy}y\,dx\\
&=\sum_{n=1}^\infty a_f(n)(2\pi n)^{-s}\Gamma(s)
\int_1^\infty\phi(x)e(\alpha nx)x^{-s}\,dx.
\end{align*}

For $j=0,1,2,\ldots$, define
functions $\phi_j=\phi_j(x,s)$ recursively by
$$
\phi_0=\phi,
\quad\phi_{j+1}=x\frac{\partial\phi_j}{\partial x}-(s+j)\phi_j.
$$
Then, by integration by parts,
$$
\int_1^\infty\phi_j(x,s)e(\alpha nx)x^{-s-j}\,dx
=-\frac{e(\alpha n)\phi_j(1,s)}{2\pi i\alpha n}
-\frac1{2\pi i\alpha n}\int_1^\infty\phi_{j+1}(x,s)e(\alpha nx)x^{-s-j-1}\,dx.
$$
Applying this iteratively $m$ times, we find
\begin{align*}
\int_1^\infty\phi(x)e(\alpha nx)x^{-s}\,dx
&=e(\alpha n)\sum_{j=0}^{m-1}\frac{\phi_j(1,s)}{(-2\pi i\alpha n)^{j+1}}\\
&+(-2\pi i\alpha n)^{-m}\int_1^{\infty}\phi_m(x,s)e(\alpha nx)x^{-s-m}\,dx.
\end{align*}
Substituting this back into the above, we have
\begin{align*}
\frac{(2\pi)^s}{\Gamma(s)}\int_0^\infty A(\alpha+iy)y^s\frac{dy}y
&=\sum_{j=0}^{m-1}\frac{\phi_j(1,s)}{(-2\pi i\alpha)^{j+1}}
L_f(s+j+1,\alpha)\\
&+(-2\pi i\alpha)^{-m}
\sum_{n=1}^{\infty}\frac{a_f(n)}{n^{s+m}}
\int_1^\infty\phi_m(x,s)e(\alpha nx)x^{-s-m}\,dx.
\end{align*}
Each of the terms in the sum over $j$ continues to an entire function
of $s$. On the other hand, it is straightforward to prove that
$\phi_m(x,s)\ll_m(1+|s|)^mx^{k-1}$.
Thus, the final sum over $n$ is holomorphic for
$\Re(s)>k-m$.  Taking $m$ arbitrarily large establishes the lemma.
\end{proof}

\begin{lemma}
\label{Fbarlemma}
Let $\alpha\in\Q^\times$ and $z=\alpha+iy$
for some $y\in\bigl(0,\frac{|\alpha|}4\bigr]$. Then
\begin{equation}
\label{fbartaylor}
\begin{aligned}
\epsilon&(-i\sqrt{N}z)^{-k}\overline{F}\!\left(-\frac1{Nz}\right)\\
&=O_{\alpha,M}(y^{M-\lfloor\frac{k+3}2\rfloor})
+\epsilon N^{-\frac{k}2}\sum_{m=0}^{M-1}
\frac{(-i\alpha)^{-m-k}}{2\pi i}\int_{\Re(s)=\frac{k}2+1}
{{s+m-k}\choose{m}}
\left(\frac{N\alpha^2}{2\pi}\right)^{s+m}\\
&\hspace{8cm}\cdot
\Gamma(s+m)
D_{\bar{f}}\left(s+m,-\frac1{N\alpha}\right)y^{-s}\,ds
\end{aligned}
\end{equation}
for every $M\in\Z_{\ge0}$.
\end{lemma}
\begin{proof}
This was essentially done in \cite[\S2]{bk-converse}; we reproduce
the argument here for the sake of completeness.
Let $z=\alpha+iy$, $\beta=-1/N\alpha$ and $u=y/\alpha$. Then
$$
-\frac1{Nz}=\beta+i|\beta u|-\frac{\beta u^2}{1+iu},
$$
so that
$$
\epsilon(-i\sqrt{N}z)^{-k}\overline{F}\!\left(-\frac1{Nz}\right)
=\epsilon (-i\sqrt{N}\alpha)^{-k}
\sum_{n=1}^\infty c_{\bar{f}}(n)e(\beta n)
e^{-2\pi n|\beta u|}
(1+iu)^{-k}e\!\left(-\frac{n\beta u^2}{1+iu}\right).
$$
Next,
\begin{align*}
(1+iu)^{-k}e\!\left(-\frac{n\beta u^2}{1+iu}\right)
&=\sum_{j=0}^\infty(-iu)^j(1+iu)^{-j-k}\frac{(-2\pi n|\beta u|)^j}{j!}\\
&=\sum_{j=0}^\infty\sum_{\ell=0}^\infty
{{j+k+\ell-1}\choose\ell}(-iu)^{j+\ell}\frac{(-2\pi n|\beta u|)^j}{j!}\\
&=\sum_{m=0}^\infty(-iu)^m\sum_{j=0}^m
{{m+k-1}\choose{m-j}}\frac{(-2\pi n|\beta u|)^j}{j!}.
\end{align*}
Note further that for any $M, K\in\Z_{\ge0}$ we have
$$
\begin{aligned}
&\hspace{-1cm}\left|\sum_{m=M}^{\infty}(-iu)^m
\sum_{j=0}^m{{m+k-1}\choose{m-j}}
\frac{(-2\pi n|\beta u|)^j}{j!}\right|\\
&\le (2\pi n|\beta u|)^{-K}K!
\sum_{m=M}^{\infty}|u|^m\sum_{j=0}^m{{m+k-1}\choose{m-j}}{{j+K}\choose{j}}
\frac{(2\pi n|\beta u|)^{j+K}}{(j+K)!}\\
&\le (\pi n|\beta u|)^{-K}K!(3/2)^{k-1}
\sum_{m=M}^{\infty}(3|u|)^me^{2\pi n|\beta u|}\\
&\ll_{\alpha,M,K}|u|^{M-K}n^{-K}e^{2\pi n|\beta u|},
\end{aligned}
$$
since $|u|\le1/4$.  Hence, substituting the definition of $u$, we have
$$
\begin{aligned}
&\epsilon(-i\sqrt{N}z)^{-k}\overline{F}\!\left(-\frac1{Nz}\right)
=O_{\alpha,M,K}\Bigl(y^{M-K}\sum_{n=1}^\infty|c_f(n)|n^{-K}\Bigr)\\
&+\epsilon(-i\sqrt{N}\alpha)^{-k}
\sum_{m=0}^{M-1}\left(\frac{-iy}{\alpha}\right)^m
\sum_{j=0}^m{{m+k-1}\choose{m-j}}
\sum_{n=1}^\infty c_{\bar{f}}(n)e(\beta n)
\frac1{j!}\left(-\frac{2\pi ny}{N\alpha^2}\right)^j
e^{-\frac{2\pi ny}{N\alpha^2}}.
\end{aligned}
$$
Choosing $K=\lfloor\frac{k-1}2\rfloor+2$,
the error term converges and gives the
estimate $O_{\alpha,M}(y^{M-K})$.

As for the other terms, we have
\begin{align*}
y^m\sum_{n=1}^\infty&c_{\bar{f}}(n)e(\beta n)
\frac1{j!}\left(-\frac{2\pi ny}{N\alpha^2}\right)^j
e^{-\frac{2\pi ny}{N\alpha^2}}
=\frac{y^{j+m}}{j!}\frac{d^j}{dy^j}
\sum_{n=1}^\infty c_{\bar{f}}(n)e(\beta n)
e^{-\frac{2\pi ny}{N\alpha^2}}\\
&=\frac{y^{j+m}}{j!}\frac{d^j}{dy^j}
\frac1{2\pi i}\int_{\Re(s)=m+\frac{k}2+1}
\left(\frac{N\alpha^2}{2\pi}\right)^s\Gamma(s)
D_{\bar{f}}(s,\beta)y^{-s}\,ds\\
&=\frac1{2\pi i}\int_{\Re(s)=\frac{k}2+1}
{{-s-m}\choose{j}}
\left(\frac{N\alpha^2}{2\pi}\right)^{s+m}\Gamma(s+m)
D_{\bar{f}}(s+m,\beta)y^{-s}\,ds.
\end{align*}
Moreover, by the Chu--Vandermonde identity we have
$$
\sum_{j=0}^m{{m+k-1}\choose{m-j}}{{-s-m}\choose{j}}
={{-s+k-1}\choose{m}}=(-1)^m{{s+m-k}\choose{m}}.
$$
Collecting these strands together, we arrive at \eqref{fbartaylor}.
\end{proof}

\begin{lemma}
\label{Blemma}
For any $\alpha\in\Q^\times$ there are numbers
$P_j(\alpha)$, $j=0,1,2,\ldots$, such that
$$
B(\alpha+iy)=\sum_{j=0}^{M-1}P_j(\alpha)y^j+O_{\alpha,M}(y^M)
$$
for all $M\in\Z_{\ge0}$ and $y\in\bigl(0,\frac{|\alpha|}4\bigr]$.
\end{lemma}
\begin{proof}
For $z=\alpha+iy$, we have
\begin{equation}
\label{izstaylor}
(-iz)^{-s}=e^{i\frac{\pi}2\sgn(\alpha)s}|\alpha|^{-s}
\left(1+\frac{iy}{\alpha}\right)^{-s}
=e^{i\frac{\pi}2\sgn(\alpha)s}|\alpha|^{-s}
\sum_{j=0}^\infty {{-s}\choose{j}}
\left(\frac{iy}{\alpha}\right)^j.
\end{equation}
Since $y\le\frac{|\alpha|}4$, the crude bound
$$
\left|{{-s}\choose{j}}\right|
=\left|{{s+j-1}\choose{j}}\right|\le 2^{|s|+j}
$$
yields
$$
\sum_{j=M}^{\infty}
{{-s}\choose{j}}\left(\frac{iy}{\alpha}\right)^j
\ll_{\alpha,M} 2^{|s|}y^M.
$$
Hence, if we truncate the sum in \eqref{izstaylor} at $M$ and substitute
it for $(-iz)^{-s}$ in the definition of $B(z)$, then
since the contour $\mathcal{C}$ is compact, the first integral
of the error term
converges to give an $O_{\alpha,M}(y^M)$ error overall.
Similarly, by standard estimates, along the line $\Re(s)=k-\frac12$
the function
$e^{i\frac{\pi}2\sgn(\alpha)s}|\alpha|^{-s}\Lambda_f(s)$
has at most polynomial growth,
and $\frac{\pi^2}{\sin^2(\pi s)}\ll e^{-2\pi|s|}$.
Since $e^{2\pi}>2$, the second integral of the error term converges as
well, and the lemma follows with
\begin{align*}
P_j(\alpha)&=
\frac1{2\pi i}\int_{\mathcal{C}}
(-i\alpha)^{-j}e^{i\frac{\pi}2\sgn(\alpha)s}|\alpha|^{-s}{{-s}\choose j}
\Delta_f(s)\,ds\\
&+\frac1{2\pi i}\int_{\Re(s)=k-\frac12}
(-i\alpha)^{-j}e^{i\frac{\pi}2\sgn(\alpha)s}|\alpha|^{-s}{{-s}\choose j}
\Lambda_f(s)\frac{\pi^2}{\sin^2(\pi s)}\,ds.
\end{align*}
\end{proof}

Now, to conclude the proof, let us define
\begin{align*}
g(y)&=F(\alpha+iy)+A(\alpha+iy)
-\sum_{j=0}^{M-1}P_j(\alpha)y^j\chi_{(0,|\alpha|/4]}(y)\\
&-\epsilon N^{-\frac{k}2}\sum_{m=0}^{M-1}
\frac{(-i\alpha)^{-m-k}}{2\pi i}\int_{\Re(s)=\frac{k}2+1}
{{s+m-k}\choose{m}}
\left(\frac{N\alpha^2}{2\pi}\right)^{s+m}\\
&\hspace{6cm}\cdot\Gamma(s+m)
D_{\bar{f}}\left(s+m,-\frac1{N\alpha}\right)y^{-s}\,ds,
\end{align*}
where $\chi_{(0,|\alpha|/4]}(y)=1$
if $y\in\bigl(0,\frac{|\alpha|}4\bigr]$
and $0$ otherwise.
Combining Lemmas~\ref{mainidlemma}, \ref{Fbarlemma} and \ref{Blemma},
we have that
$g(y)=O_{\alpha,M}(y^{M-\lfloor\frac{k+3}2\rfloor})$
for $y\in\bigl(0,\frac{|\alpha|}4\bigr]$.
On the other hand, it is easy to see that $g(y)$ decays rapidly
as $y\to\infty$. Thus,
$\frac{(2\pi)^s}{\Gamma(s)}\int_0^\infty g(y)y^{s-1}\,dy$
defines a holomorphic function for
$\Re(s)>\lfloor\frac{k+3}2\rfloor-M$.

Note that
$$
\frac{(2\pi)^s}{\Gamma(s)}\int_0^{\infty}\sum_{j=0}^{M-1}P_j(\alpha)
y^j\chi_{(0,|\alpha|/4]}(y)y^s\frac{dy}y
=\frac{(2\pi)^s}{\Gamma(s)}\sum_{j=0}^{M-1}P_j(\alpha)
\frac{(|\alpha|/4)^{s+j}}{s+j}
$$
extends to an entire function of $s$.
Together with \eqref{Fmellin} and Lemma~\ref{Alemma},
this shows that
\eqref{maineq} is holomorphic for
$\Re(s)>\lfloor\frac{k+3}2\rfloor-M$.
Finally, we replace $M$ by $M+1$ and discard the final term
of the sum over $m$ to see that \eqref{maineq}
is in fact holomorphic for $\Re(s)>\frac{k+1}2-M$.
\qed

\bibliographystyle{amsplain}
\providecommand{\bysame}{\leavevmode\hbox to3em{\hrulefill}\thinspace}
\providecommand{\MR}{\relax\ifhmode\unskip\space\fi MR }
\providecommand{\MRhref}[2]{%
  \href{http://www.ams.org/mathscinet-getitem?mr=#1}{#2}
}
\providecommand{\href}[2]{#2}

\end{document}